\documentclass[11pt, letterpaper]{amsart}
\pdfoutput=1 
\usepackage{amssymb}
\usepackage{graphicx}
\usepackage[shortlabels]{enumitem}
\usepackage{mathtools} 
\usepackage{xcolor}
\usepackage{xspace}
\usepackage[normalem]{ulem} 
\usepackage[toc]{appendix}

\definecolor{darkblue}{rgb}{0,0,0.4} 
\usepackage{xr-hyper}
\usepackage[colorlinks=true, citecolor=darkblue, filecolor=darkblue, linkcolor=darkblue,urlcolor=darkblue]{hyperref} 
\usepackage[all]{hypcap}

\usepackage{tikz}
\usepackage{caption}

\newcommand{\ILtikzpic}[2][]{
\vcenter{\hbox{\begin{tikzpicture}[#1]
#2
\end{tikzpicture}}}
}

\newtheorem{thm}{Theorem}[section]
\newtheorem{theorem}[thm]{Theorem}

\newtheorem{corollary}[thm]{Corollary}

\newtheorem{lemma}[thm]{Lemma}

\newtheorem{proposition}[thm]{Proposition}

\newtheorem{conjecture}[thm]{Conjecture}

\newtheorem{question}[thm]{Question}
\theoremstyle{definition}
\newtheorem{answer}[thm]{Answer}

\newtheorem{definition}[thm]{Definition}

\theoremstyle{remark}

\newtheorem{remark}[thm]{Remark}

\numberwithin{equation}{section}

\newcommand{\mc}{\mathcal}

\renewcommand{\theta}{\vartheta}

\newcommand{\abs}[1]{\left\vert#1\right\vert}
\newcommand{\set}[1]{\left\{#1\right\}}

\DeclareMathOperator{\rk}{rk}

\DeclareMathOperator{\lk}{lk}

\newcommand{\de}{\partial}

\newcommand{\Z}{\mathbb{Z}}
\newcommand{\Q}{\mathbb{Q}}

\newcommand{\F}{\mathbb{F}}

\newcommand{\ring}{\mathcal{A}}
\mathchardef\mhyphen="2D

\newcommand{\gr}{\mathbf{gr}}
\newcommand{\w}{\mathbf{w}}
\newcommand{\z}{\mathbf{z}}

\DeclareMathOperator{\Arf}{Arf}

\newcommand{\Alex}{\mathcal{A}}
\newcommand{\Jones}{V}
\newcommand{\J}{V}
\newcommand{\V}{V}

\DeclareMathOperator{\Ord}{Ord}
\DeclareMathOperator{\HF}{HF}
\DeclareMathOperator{\HFK}{HFK}
\DeclareMathOperator{\CFK}{CFK}
\DeclareMathOperator{\Wh}{Wh}
\DeclareMathOperator{\Kh}{Kh}
\newcommand{\HFh}{\widehat{\HF}}
\newcommand{\HFKh}{\widehat{\HFK}}
\newcommand{\HFKm}{\HFK^-}
\newcommand{\rKh}{\widetilde{\Kh}}
\newcommand{\xo}{\mathrm{xo}}

\begin{document}

\title[On the rank of knot homology theories and concordance]{On the rank of knot homology theories\\ and concordance}

\author[Dunfield]{Nathan M. Dunfield}%
\address{Department of Mathematics, University of Illinois at Urbana-Champaign, Urbana, IL 61801, USA}%
\email{\href{mailto:nathan@dunfield.info}{nathan@dunfield.info}}

\author[Gong]{Sherry Gong}%
\address{Department of Mathematics, Texas A \& M University, College Station, TX 77840, USA}%
\email{\href{mailto:sgongli@tamu.edu}{sgongli@tamu.edu}}

\author[Hockenhull]{Thomas Hockenhull}%
\address{Frome, Somerset, UK}%
\email{\href{mailto:thomas.hockenhull@gmail.com}{thomas.hockenhull@gmail.com}}

\author[Marengon]{Marco Marengon}%
\address{Alfr\'ed R\'enyi Institute of Mathematics, Budapest, Hungary}%
\email{\href{mailto:marengon@renyi.hu}{marengon@renyi.hu}}%

\author[Willis]{Michael Willis}
\address{Department of Mathematics, Texas A \& M University, College Station, TX 77840, USA}
\email{\href{mailto:msw188@ucla.edu}{msw188@tamu.edu}}

\begin{abstract}
%
%
%
For a ribbon knot, it is a folk conjecture that the rank of its knot Floer homology must be 1 modulo 8, and another folk conjecture says the same about reduced Khovanov homology.  We give the first counter-examples to both of these folk conjectures, but at the same time present compelling evidence for new conjectures that either of these homologies must have rank congruent to 1 modulo 4 for any ribbon knot.  We prove that each revised conjecture is equivalent to showing that taking the rank of the homology modulo 4 gives a homomorphism of the knot concordance group.  We check the revised conjectures for 2.4~million ribbon knots, and also prove they hold for ribbon knots with fusion number~1.
\end{abstract}

\maketitle

%

\vspace{-1cm}

\tableofcontents

%

\section{Introduction}

Homological invariants of knots have played a prominent role in low-dimensional topology in the last two decades: among them are Khovanov homology \cite{K:categorification} and knot Floer homology \cite{OSz:HFK, R:HFK}.  For each, one can take a graded Euler characteristic and recover a classical polynomial invariant of knots (the Jones polynomial and the Alexander polynomial, respectively).
Moreover, several concordance invariants and genus bounds arise from these homology theories, often via spectral sequences \cite{OSz:tau, R:s, OSSz:Upsilon}.  On the other hand, results on the ranks of these homology theories are elusive, and often focused on determining which knots can exhibit homology of a given rank.  (See for instance \cite{KM:detector} and \cite[Theorem 1.2]{OSz:HFK} for unknot detection results.) One popular folk conjecture, however, concerns the rank modulo 8
of each homology.  Let us denote
\begin{enumerate}
    \item the hat flavour of knot Floer homology with $\F_2$ coefficients by $\HFKh$; and
    \item the reduced Khovanov homology with $\ring$ coefficients by $\rKh_\ring$.
\end{enumerate}
A version of the folk conjecture is:

\begin{conjecture}[1 mod 8 Folk Conjecture]\label{conj:ribbon mod 8}

  
Let $\ring$ be $\F_p$ or $\Q$. Given a ribbon knot $R \subset S^3$, we have:
\begin{enumerate}[(a)]
    \item \label{it:r-HFK8} $\rk\HFKh(R) \equiv 1 \pmod 8$; and
    \item \label{it:r-rKh8} $\rk\rKh_\ring(R) \equiv 1 \pmod 8$.
\end{enumerate}

\end{conjecture}

This folk conjecture was motivated in part by attempts to give a categorified version of the Fox-Milnor theorem, which implies that slice knots $K\subset S^3$ have $\det K$ an odd square and hence is $1 \pmod 8$, together with a large body of experimental evidence.  However, we show that the folk conjecture is false:

\begin{theorem}\label{thm:1 mod 8 false}
The 1 mod 8 Folk Conjecture \ref{conj:ribbon mod 8} is false for both $\HFKh$ and for $\rKh_\ring$ with $\ring=\F_p$ for $p \neq 2$ and $\ring = \Q$.
\end{theorem}

We prove Theorem \ref{thm:1 mod 8 false} in Section \ref{sec:Computations} by producing explicit ribbon knots whose homology ranks are computed and shown not to be congruent to 1 modulo 8.  These counter-examples were found while conducting extensive experimental testing of Conjecture \ref{conj:ribbon mod 8}.  Within this large dataset, all such counter-examples had rank $5\pmod 8$; we are therefore led to the following weaker form of Conjecture~\ref{conj:ribbon mod 8}:

\begin{conjecture}
\label{conj:ribbon}
Let $\ring$ be $\F_p$ or $\Q$. Given a ribbon knot $R \subset S^3$, we have:
\begin{enumerate}[(a)]
    \item \label{it:r-HFK} $\rk\HFKh(R) \equiv 1 \pmod 4$;
    \item \label{it:r-rKh} $\rk\rKh_\ring(R) \equiv 1 \pmod 4$.
\end{enumerate}
\end{conjecture}

The following conjecture then seems like a generalization of Conjecture~\ref{conj:ribbon}, but in fact Theorem~\ref{thm:equivalence} will show that they are equivalent.

\begin{conjecture}
\label{conj:mod4}
Let $\ring$ be $\F_p$ or $\Q$.
Given a knot $K \subset S^3$, the following numbers are concordance invariants of $K$:
\begin{enumerate}[(a)]
    \item \label{it:m-HFK} $\rk \HFKh(K) \pmod 4$;
    \item \label{it:m-rKh} $\rk\rKh_\ring(K) \pmod 4$.
\end{enumerate}
Thus, there are homomorphisms from the concordance group
\begin{align*}
\rho_{\HFKh} \colon \mc C & \to (\Z/4\Z)^* &
\rho_{\rKh_\ring} \colon \mc C & \to (\Z/4\Z)^* \\
K & \mapsto [\rk \HFKh(K)]_4 &
K & \mapsto [\rk \rKh_\ring(K)]_4
\end{align*}
\end{conjecture}

\begin{theorem}
\label{thm:equivalence}
  
We have the following equivalences:
\begin{enumerate}[(a)]
    \item Conjecture \ref{conj:mod4}\ref{it:m-HFK} $\Leftrightarrow$ Conjecture \ref{conj:ribbon}\ref{it:r-HFK};
    \item Conjecture \ref{conj:mod4}\ref{it:m-rKh} $\Leftrightarrow$ Conjecture \ref{conj:ribbon}\ref{it:r-rKh}.
\end{enumerate}
\end{theorem}

In particular, the word ``ribbon'' in Conjecture \ref{conj:ribbon} can be replaced by the words ``smoothly slice''.

We note that none of our data produced a counter-example for the Folk Conjecture \ref{conj:ribbon mod 8}\ref{it:r-rKh8} over $\ring=\F_2$, leading to:

\begin{question}\label{qu:Kh 1 mod 8 over F2}
    Is the Folk Conjecture \ref{conj:ribbon mod 8}\ref{it:r-rKh8} true with coefficients $\ring=\F_2$?
\end{question}

\subsection{Evidence for our conjectures}

We offer two kinds of evidence for our conjectures.

\subsubsection{Empirical evidence}
\label{sec:empirical}
We tested Conjecture \ref{conj:ribbon} on a list of 1.6 million ribbon knots of crossing number at most 19 from \cite{DunfieldGong2023}. We also generated 800,000 knots of up to 30 crossings by taking symmetric unions (a special family of ribbon knots defined by Lamm \cite{Lamm:SymUnions}, building on a construction of Kinoshita-Terasaka \cite{KT:SymUnions}; see Figure \ref{fig:gen sym union}) with one twisting region. For the combined sample of 2.4 million ribbon knots, we found Conjecture \ref{conj:ribbon}\ref{it:r-HFK} holds for them using Szab\'o's knot Floer calculator \cite{Sz:calculator} through SnapPy \cite{SnapPy}.  Conjecture~\ref{conj:ribbon}\ref{it:r-rKh} also holds for these knots  and with $\ring = \F_2,\F_3,$ and $\F_{211}$ (see Remark \ref{rmk:base fields}). This was tested using KnotJob \cite{KnotJob}.

\begin{figure}
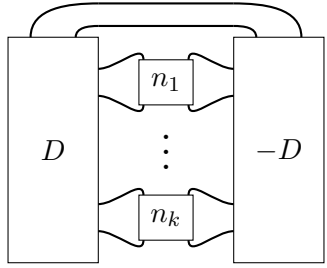

\[
\ILtikzpic[scale=.3]{
\draw (-1.2,2) rectangle (1.2,4);
\node at (0,3) {$n_1$};
\node[scale=1.4] at (0,.4) {$\vdots$};
\draw (-1.2,-2) rectangle (1.2,-4);
\node at (0,-3) {$n_k$};
\draw[thick] (1,2) to[out=-90,in=180] (3,2.4);
\draw[thick] (-1,2) to[out=-90,in=0] (-3,2.4);
\draw[thick] (1,4) to[out=90,in=180] (3,3.6);
\draw[thick] (-1,4) to[out=90,in=0] (-3,3.6);
\draw[thick] (1,-2) to[out=90,in=180] (3,-2.4);
\draw[thick] (-1,-2) to[out=90,in=0] (-3,-2.4);
\draw[thick] (1,-4) to[out=-90,in=180] (3,-3.6);
\draw[thick] (-1,-4) to[out=-90,in=0] (-3,-3.6);
\draw (3,-5) rectangle (7,5);
\draw (-3,-5) rectangle (-7,5);
\node at (5,0) {$-D$};
\node at (-5,0) {$D$};
\draw[thick] (4,5) to[out=90,in=0] (3,5.5);
\draw[thick] (6,5) to[out=90,in=0] (3,6.5);
\draw[thick] (-4,5) to[out=90,in=180] (-3,5.5);
\draw[thick] (-6,5) to[out=90,in=180] (-3,6.5);
\draw[thick] (-3,5.5)--(3,5.5);
\draw[thick] (-3,6.5)--(3,6.5);
}
\]
    \caption{A symmetric union of a knot diagram $D$.  The various $n_i$ indicate twisting regions with $n_i$ (signed) crossings concatenated vertically.  All symmetric unions are ribbon; the converse is currently unknown.}
    \label{fig:gen sym union}
\end{figure}

\subsubsection{Fusion number 1 evidence}
Building on the work of Hom-Kang-Park \cite{HKP:ribbon}, we prove that the stronger 1 mod 8 Folk Conjecture \ref{conj:ribbon mod 8}\ref{it:r-HFK} holds for ribbon knots with fusion number $1$ (i.e.\ those knots obtained by attaching a band to a 2-component unlink).

\begin{theorem}
\label{thm:fusion1}
If $R$ is a ribbon knot with fusion number $1$, then
\[
\rk \HFKh(R) \equiv 1 \pmod8,
\]
in particular, $R$ satisfies Conjecture \ref{conj:ribbon}\ref{it:r-HFK}.
\end{theorem}

The result holds more generally for ribbon knots with torsion order $1$ (cf.\ Theorem \ref{thm:TO}). The key ingredient of the proof is Hom's decomposition for $\CFK^\infty$ of a slice knot (cf.\ \cite[Theorem 1]{H:survey}).

Meanwhile for Khovanov homology, in unpublished work Robert Lipshitz and Sucharit Sarkar showed the following, whose proof we include in Section~\ref{sec:Kh xo 1} with their permission.

\begin{theorem}[Lipshitz-Sarkar]
\label{thm:fusion1 Kh}
If $R$ is a ribbon knot with fusion number 1 and $\ring=\Q$ or $\F_p$ for $p\neq 2$, then
\[\rk \rKh_\ring(R)\equiv 1 \pmod 4,\]
satisfying Conjecture \ref{conj:ribbon}\ref{it:r-rKh}.
\end{theorem}

Although neither of the proofs of Theorems \ref{thm:fusion1} and \ref{thm:fusion1 Kh} seem to generalize to give a statement about Khovanov homology mod 8, it is perhaps conspicuous that every counter-example to the 1 mod 8 Folk Conjecture \ref{conj:ribbon mod 8} that we have found is fusion number 2.  In analogy with Theorem \ref{thm:fusion1}, we pose the following as a question.

\begin{question}\label{qu:Khr fusion number 1}
If $R$ is a ribbon knot with fusion number $1$, must we have
\[
\rk \rKh_\ring(R) \equiv 1 \pmod8
\]
for some (all?) $\ring$?
\end{question}

\subsection{Comparisons with some known slice obstructions}

\subsubsection{The $s$-invariant} If Conjecture~\ref{conj:mod4}\ref{it:r-rKh} holds, it would give the obstruction that any knot with $\rk \rKh_\ring \not\equiv 1 \pmod 4$ is not smoothly slice. It is natural to ask how this obstruction compares to existing techniques, especially the Rasmussen $s$-invariant \cite{R:s}, which derives from the same underlying knot homology theory.  Among the 3.8 million knots considered in \cite{DunfieldGong2023}, for $\ring = \Q$, there are knots where $s \neq 0$ and $\rk \rKh_\ring \equiv 1 \pmod 4$ and others where $s = 0$ and $\rk \rKh_\ring \not\equiv 1 \pmod 4$, showing that neither condition implies the other.  However, for more than 99\% of these knots, the two tests agree, with $s \neq 0$ occurring slightly more frequently than $\rk \rKh_\ring \not\equiv 1 \pmod 4$.  Nearly all knots in this sample have $\abs{s} \leq 2$; as $s$ grows, we expect the $s \neq 0$ test to be generically the stronger of the two.

The $s$-invariant is a particularly subtle concordance invariant that gives strong information about the smooth 4-ball genus.  It can be used to show certain topologically slice knots are not smoothly slice, for example the Piccirillo knot $K_P$ from \cite{P:Conway}, a fact which was leveraged there to show the Conway knot is not smoothly slice.  Using \cite{LL:Khoca}, we found that $\rk \rKh_{\F_3}(K_P) = \rk \rKh_\Q(K_P) = 263 \equiv 3 \pmod 4$.  Thus, if Conjecture~\ref{conj:mod4}\ref{it:r-rKh} holds, the resulting obstruction would be smooth rather than topological in nature and sensitive enough to handle $K_P$.  It could thus be potentially applied in place of the $s$-invariant in the strategy of \cite{FGMW:Poincare} and \cite{MP:Poincare} to produce counterexamples to the smooth 4-dimensional Poincar\'e conjecture.

\subsubsection{The $\tau$-invariant}

Similarly, if Conjecture~\ref{conj:mod4}\ref{it:r-HFK} holds, it would give the obstruction that any knot with $\rk \HFKh \not\equiv 1 \pmod 4$ is not smoothly slice.  Here, a natural point of comparison is the $\tau$ invariant from \cite{OSz:tau}, which is also defined in terms of $\HFKh$.  For the sample of \cite{DunfieldGong2023}, we again found that this test would match the condition $\tau \neq 0$ for 99\% of these knots, though it is also not the case that either test would subsume the other.  It is known that $\tau$ is unsuitable for applications such as \cite{P:Conway}, and so it is unsurprising that the Piccirillo knot has $\rk \HFKh(K_p) = 33 \equiv 1 \pmod 4$.  However, as with $\tau$, it is the case that $\rk \HFKh \not\equiv 1 \pmod 4$ would be a genuine smooth slice obstruction. For example, there are almost 6,000 knots in \cite{DunfieldGong2023} with trivial Alexander polynomial (and hence topologically slice \cite{F:Poincare}) but where $\tau \neq 0$ and $\rk \HFKh \equiv 3 \pmod 4$.  We give infinitely many examples along these lines in Section~\ref{sec:Whitehead}.




\subsubsection{Knots with order $2$ in the concordance group}
If either part of Conjecture \ref{conj:mod4} holds, it would give a sliceness obstruction amenable to detecting 2-torsion in the concordance group.
While there already exist obstructions that detect 2-torsion (for example, the Fox-Milnor obstruction \cite{FM:Fox-Milnor}), 2-torsion cannot be detected with homomorphisms $\mathcal C \to \Z$, which typically arise in the context of knot homologies.

\subsection{Non-locality}
\label{sec:intro non-locality}

A subtlety of our conjecture is that there cannot exist a local proof, i.e.\ a proof using only a projection of the knot near a ribbon singularity of an immersed ribbon disc.
We show this in Section \ref{sec:non-locality}; for a precise statement, see Proposition \ref{prop:non-locality}.

\subsection{Generalisations}

\subsubsection{Unreduced Khovanov homology}
For the case of unreduced Khovanov homology $\Kh_\ring$, we have corresponding conjectures.
\begin{conjecture}
\label{conj:ribbon unreduced}
Let $\ring$ be either $\Q$ or $\F_p$ for $p\neq 2$.  Given a ribbon knot $R\subset S^3$, we have
\[\rk\Kh_\ring(R) \equiv 2 \pmod 4.\]
\end{conjecture}

\begin{conjecture}
\label{conj:mod4 unreduced}
Let $\ring$ be either $\Q$ or $\F_p$ for $p\neq 2$.  Given a knot $K\subset S^3$, the value
\[\rk \Kh_\ring(K) \pmod 4\]
is a concordance invariant of $K$.
\end{conjecture}

Note that in this setting, we do not know if these conjectures are equivalent. In line with Theorem~\ref{thm:1 mod 8 false},
it is not possible to strengthen Conjecture~\ref{conj:ribbon unreduced} to a statement modulo 8.  In particular, there are ribbon knots where $\rk \Kh_\Q \equiv 2 \pmod 8$ and others where $\rk \Kh_\Q \equiv 6 \pmod 8$, for example $6_1 = K6a3$ and $K16n33000$ respectively.

\begin{remark}\label{rmk:Kh over F2}
  In the case that $\ring=\F_2$, we have $\rk\Kh_{\F_2}(K)=2\rk\rKh_{\F_2}(K)$, cf.\ \cite[Corollary 3.2.C]{S:torsion}.  Therefore, since $\rk\rKh_{\F_2}(K)$ is always odd, we immediately have $\rk\Kh_{\F_2}(K)\equiv 2\pmod 4$ for any knot $K$, rendering the $\F_2$ analogue of Conjecture~\ref{conj:ribbon unreduced} uninteresting.  This same relationship between $\rk\Kh_{\F_2}$ and $\rk\rKh_{F_2}$ also implies that our Conjectures \ref{conj:ribbon}\ref{it:r-rKh} and~\ref{conj:mod4}\ref{it:m-rKh} are equivalent to corresponding conjectures for unreduced Khovanov homology modulo 8.
\end{remark}

We checked that Conjecture \ref{conj:ribbon unreduced} holds for all the 2.4 million ribbon knots in the list mentioned in Section \ref{sec:empirical}, with coefficient ring $\F_3$ and $\F_{211}$.


\subsubsection{3-manifolds}
The simplest 3-manifold analogue of our conjecture for Heegaard Floer homology would be the following.

\begin{question}\label{qu:ZHS3 bounding ZHB4}
If a $\Z HS^3$ $Y$ bounds a $\Z HB^4$, is it true that
\[
\rk \HFh (Y) \equiv 1 \pmod 4 ?
\]
\end{question}

A positive answer to this question would be extremely interesting, because the resulting obstruction would be suitable to detect 2-torsion in the integer cobordism group $\Theta_\Z^3$.
Whether such 2-torsion exists is currently an open question. (The Rokhlin invariant is not a helpful obstruction here, because it vanishes on the 2-torsion subgroup of $\Theta_\Z^3$, cf.\ \cite{M:triangulation}.)


\subsection{Organisation}
In Section \ref{sec:motivation}, we give motivations for the Conjecture~\ref{conj:ribbon}, whose potential resulting obstructions to sliceness are then shown to be smooth in Section \ref{sec:applications}.
In Section \ref{sec:equivalence}, we prove Theorem \ref{thm:equivalence} about the equivalence of Conjectures \ref{conj:ribbon} and \ref{conj:mod4}.
In Section \ref{sec:Computations}, we discuss the computations that disprove the Folk Conjecture \ref{conj:ribbon mod 8}.
In Section \ref{sec:fusion1}, we focus on ribbon knots with fusion number 1, for which Conjecture \ref{conj:ribbon} holds.
In Section \ref{sec:non-locality}, we show that a general proof of Conjecture \ref{conj:ribbon} cannot be attained by simple local arguments on the knot diagram.
In Appendix \ref{appendix}, we discuss the (absence of a) relation with the Arf invariant.  Finally, in Appendix~\ref{app: PD codes}, we give descriptions of some of the counter-examples in Section~\ref{sec:Computations}.

\subsection{Acknowledgements}
We are especially grateful to Adam S.\ Levine, Sucharit Sarkar, and Andr\'as I.\ Stipsicz for stimulating conversations and for sharing their ideas.
We also would like to thank Jennifer Hom, Patrick Gilmer, Marco Golla, Lisa Piccirillo, and Arunima Ray.
We used the following websites and computer programs for preliminary computations: SnapPy \cite{SnapPy}, the Knot Atlas \cite{KnotAtlas},  Khoca \cite{LL:Khoca}, KnotInfo \cite{knotinfo}, the program written for \cite{SS:two-fold}, KnotJob \cite{KnotJob}, and Knot Floer homology calculator \cite{Sz:calculator}. We are very grateful to their developers.

Dunfield was partially supported by US National Science Foundation grant DMS-1811156 and a Simons fellowship.
Gong and Willis were supported by NSF Grant No.\ 1440140 while in residence at the Simons Laufer Mathematical Sciences Institute in Berkeley, California, during Fall 2022.
Gong was supported by US National Science Foundation grant DMS-2055736.
Hockenhull was supported by the Max Planck Institute for Mathematics for much of the period in which this research was undertaken.
Marengon acknowledges that:
This project has received funding from the European Union’s Horizon 2020 research and innovation programme under the Marie Sk{\l}odowska-Curie grant agreement No.\ 893282.

\section{Background}
\label{sec:motivation}
The original motivation for our conjecture comes from the Fox-Milnor theorem. This states that if $K \subset S^3$ is a slice knot, then its Alexander polynomial admits a factorisation
\[
\Delta_K(t) \doteq f(t) \cdot f(t^{-1})
\]
for some $f \in \Z[t,t^{-1}]$.
(The symbol `$\doteq$' denotes `equal up to multiplication by a unit in the ring $\Z[t, t^{-1}]$.)
If $K$ is ribbon, then one can choose as $f(t)$ the Alexander polynomial of a ribbon disc $D \subset B^4$ for $K$.

The na\"{i}ve categorification of the Fox-Milnor theorem would state that, for a slice knot $K$, $\HFKh(K)$ factors, as a bigraded $\F_2$-vector space, as $V \otimes V^{*}$, where $V^*$ denotes the dual of the vector space $V$. Obviously, this would imply that if $K$ is slice then $\rk \HFKh(K)$ must be an odd square. This is famously false, as for the Kinoshita-Terasaka knot (which is slice, and in fact also ribbon and a symmetric union) one has $\rk \HFKh = 33$. 

Nonetheless, for some families of slice knots one can still show that $\rk \HFKh$ is an odd square:
\begin{itemize}
    \item if $K$ is a quasi-alternating slice knot, then $\HFKh(K) = V \otimes V^*$ by \cite{MO:quasi-alternating};
    \item if $K$ is the connected sum of a knot and its mirror (these knots are called \emph{rectangular} in \cite{HMW1}), then $\HFKh(K) = V \otimes V^*$ by the connected sum formula for $\HFKh$, cf.\ \cite[Equation (5)]{OSz:HFK} and \cite[Proposition 6.2]{R:HFK};
    \item if $K$ is a symmetric union in the family described in \cite{HMW1}, then $\rk \HFKh (K)$ is an odd square, cf.\ \cite{HMW1}.
\end{itemize}
However, with the Kinoshita-Terasaka knot being an inconvenient counterexample to any odd square rank conjecture, one is led to the question of what the number $33$ has in common with odd squares. One answer is that all these numbers are congruent to $1 \pmod 8$.  The ranks of the reduced Khovanov homology of the knots above are also odd squares, and thus one is led to the Folk Conjecture~\ref{conj:ribbon mod 8}.  However, as we will show in Section \ref{sec:Computations}, there exist ribbon knots whose homology ranks are congruent to $5 \pmod 8$, and thus we formulate Conjecture \ref{conj:ribbon} and the equivalent Conjecture \ref{conj:mod4}.

\begin{remark}
We remark that one cannot hope to strengthen the conjecture beyond $\mathrm{mod}\, 4$.
If $H$ denotes $\HFKh$ or $\rKh_\ring$, and the $\mathrm{mod} \, m$ reduction of $\rk H(\cdot)$ is a concordance invariant, then $m|8$.  This follows from the fact that the square knot $S := T_{2,3} \# \overline{T_{2,3}}$ is slice and $\rk H(S) = 9$, so one would have $9 \equiv 1 \pmod m$, which forces $m|8$.  Then since $m=8$ itself is ruled out via Theorem \ref{thm:1 mod 8 false}, the choice of $m = 4$ is optimal.
\end{remark}

\begin{remark}
The careful reader will have noted that we use various coefficient rings in our conjectures and in computational experimentation. A few remarks to keep in mind are the following:
\begin{itemize}
\item $\HFKh_\ring$ can be defined (at least non-functorially) with $\ring = \F_p$, $\Q$, or $\Z$ coefficients. However, no torsion in $\HFKh_{\Z}$ has been found so far, and therefore the ranks of $\HFKh_\ring$, for $\ring = \F_p$ and $\ring = \Q$, are conjecturally the same. This is why we only consider $\HFKh = \HFKh_{\F_2}$.
\item As a module over $\Z$, (unreduced or reduced) Khovanov homology often has torsion. By the Universal Coefficient Theorem, $\rk \HFKh_\Q$ is equal to the rank of the torsion-free part of $\rk \HFKh_\Z$.
\item In some places, for our computations of Khovanov homology, we computed over $\F_{211}$ instead of over $\Q$, for reasons of computational complexity. Where we found counter-examples to Conjecture~\ref{conj:ribbon mod 8}\ref{it:r-rKh8}, this was done by finding them as counter-examples over $\F_{211}$ and then checking that they were also counter-examples over $\Q$.
\end{itemize}
\label{rmk:base fields}
\end{remark}

\section{Untwisted Whitehead doubles}
\label{sec:applications}
\label{sec:Whitehead}

Since the Piccirillo knot is topologically slice, the potential obstruction coming from $\rk\rKh_{\Q}$ would be smooth and not topological. The following examples show that the corresponding obstruction from $\rk \HFKh$ would also be smooth.

A long-standing question of Kirby (cf.\ \cite[Problem 1.38]{K:list}) asks whether a knot $K$ is slice if and only if its untwisted positive Whitehead double $\Wh^+_0(K)$ is slice. (By mirroring, one can replace `positive' with `negative'.) The subtlety of this question lies in the fact that $\Wh^+_0(K)$ is always topologically slice.

Hedden computed the full knot Floer homology of Whitehead doubles in \cite{H:Whitehead}. From his theorem \cite[Theorem 1.2]{H:Whitehead} it easily follows that
\begin{itemize}
    \item if $\tau(K)\leq0$ (for example, if $K$ is slice or is a negative knot), then $\rk(\HFKh(\Wh^+_0(K))) \equiv 1 \pmod 4$;
    \item if $\tau(K)>0$ (e.g., if $K$ is positive knot), then $\rk(\HFKh(\Wh^+_0(K))) \equiv 3 \pmod 4$.
\end{itemize}
Thus, a proof of Conjecture \ref{conj:mod4}\ref{it:m-HFK} would not give a negative answer to Kirby's question. On the other hand, it would recover Hedden's result that $Wh^+_0(K)$ is topologically but not smoothly slice if $\tau(K)>0$ (cf.\ \cite[Corollary 1.6]{H:Whitehead}).

\section{Equivalence of the two conjectures}
\label{sec:equivalence}

This section proves Theorem \ref{thm:equivalence} regarding the equivalence of Conjectures~\ref{conj:ribbon} and  \ref{conj:mod4}.

\begin{proof}[Proof of Theorem \ref{thm:equivalence}]
The forward direction of both statements is clear given that, for the unknot $U\subset S^3$, $\rk(\HFKh(U))=\rk(\rKh_{\ring}(U))=1$.

For the other direction, we let $H(\cdot)$ denote either knot homology theory $\HFKh(\cdot)$ or $\rKh_{\ring}(\cdot)$, and we assume that, for any ribbon knot $R$, we have
\[\rk(H(R))\equiv 1 \pmod{4}.\]
Now $H(\cdot)$ satisfies a connected sum formula (cf.\ \cite[Equation (2)]{OSz:HFK} and \cite[Proposition 6.2]{R:HFK} for $\HFKh$, and \cite[Proposition~3.3 and Section~3.1]{K:reduced} for $\rKh$), and thus the map
\begin{align*}
\tilde\rho \colon \set{\textrm{knots}} & \to (\Z/4\Z)^* \\ 
K & \mapsto [\rk H(K)]_4 &
\end{align*}
defines a monoid homomorphism. We will prove that under our assumptions $\rk(H(S))\equiv 1 \pmod{4}$ for every slice knot $S$, so the maps will descend to (monoid, hence group) homomorphisms from the concordance group. (If $G$ and $H$ are groups, monoid homomorphisms and group homomorphisms from $G$ to $H$ are the same notion.)

Given a slice knot $S$, there exists some ribbon knot $R$ such that the connected sum $S \# R$ is ribbon (see e.g.\ \cite[Lemma 2.5]{T:slice}), and thus
\[\rk(H(S\# R))\equiv 1 \pmod 4.\]
By the connected sum formula, we simplify the above expression to
\[\rk(H(S)) \cdot \rk(H(R)) \equiv 1 \pmod 4.\]
We have assumed that $\rk(H(R))\equiv 1 \pmod 4$, and thus $\rk(H(S))\equiv 1 \pmod{4}$ as desired.
\end{proof}

\section{Computations and proof of Theorem \ref{thm:1 mod 8 false}}
\label{sec:Computations}

In our computations, we looked at two collections of ribbon knots.  The first was a list of 1.6 million ribbon knots of up to 19 crossings that was compiled by two of the authors, which will appear alongside \cite{DunfieldGong2023}.  (There are 350 million prime knots with at most 19 crossings \cite{Burton}, and the list from \cite{DunfieldGong2023} is nearly complete, with it missing at most 20,000 ribbon knots.) Additionally, we wrote a program to build ribbon knots by taking a randomly-generated knot diagram and forming its symmetric union with a single twisting region containing a single half-twist. We generated 800,000 such knots of up to 29 crossings in this way.

On these 2.4 million knots, we found the ranks of $\HFKh$ using Szab\'o's knot Floer calculator \cite{Sz:calculator} through SnapPy \cite{SnapPy}.  We also computed the ranks of $\rKh_{\F_2}, \rKh_{\F_3}$, and $\rKh_{\F_{211}}$ (see Remark \ref{rmk:base fields}) using KnotJob \cite{KnotJob}.  For all 2.4 million knots, the resulting ranks satisfied Conjecture \ref{conj:ribbon}.  

In the case of the ribbon knots from \cite{DunfieldGong2023}, it should be noted that at least 80\% have fusion number 1 and hence are covered by Theorems~\ref{thm:fusion1} and \ref{thm:fusion1 Kh}; we suspect that most of the remaining 20\% have fusion number 2 or more, but can prove this for less than 20,000 of these knots. 
This prompted us to generate the second list of 800,000 ribbon knots, which typically have higher crossing number.
While thus the effective sample size for Conjecture~\ref{conj:ribbon} is 1.2 million at best, the full sample speaks towards a positive answer to Question~\ref{qu:Khr fusion number 1}. In addition, all of our knots satisfied the stronger Folk Conjecture \ref{conj:ribbon mod 8}\ref{it:r-rKh8} for $\rKh_{\F_2}$.  

We now present further details for the Khovanov homology and knot Floer homology separately, showcasing the counter-examples to the Folk Conjecture \ref{conj:ribbon mod 8}.

\subsection{Khovanov homology computations}
Of the 1.6 million ribbon knots of up to 19 crossings, we found four where the reduced Khovanov homology had rank $5 \pmod{8}$ over both $\ring=\F_3$ and $\Q$, violating the Folk Conjecture \ref{conj:ribbon mod 8}\ref{it:r-rKh8}.
As the integral homology $\Kh_\Z$ for each of these knots has only 2-torsion, the Universal Coefficient Theorem implies that $\rk \rKh_{\F_p} = \rk \rKh_{\Q}$ for all $p \neq 2$, proving Theorem \ref{thm:1 mod 8 false} for the setting of Khovanov homology.  One of these counter-examples is the knot 18nh\_00159590, in the notation of \cite{Burton}. This knot is depicted in Figure \ref{fig:18nh_00159590}. The other counter-examples we found are the knots 18nh\_00752242, 19nh\_000129633, 19nh\_000305767; planar diagram (PD) codes for all four knots are in Appendix~\ref{app: PD codes}.

We remark that all of these counter-examples have fusion number two (prompting us to pose Question \ref{qu:Khr fusion number 1}).  In all four cases, the computer program of \cite{DunfieldGong2023} was able to find a two band simplification of these knots to the 3-component unlink, as in Figure \ref{fig:18nh_00159590}, providing an upper bound on the fusion number.  The lower bound of 2 could then be established computationally using an obstruction of Hom-Kang-Park \cite[Proposition 3.6]{HKP:ribbon} that will be discussed further in Section \ref{sec:HFK tor order 1}.

\begin{figure}
    \centering
    \includegraphics{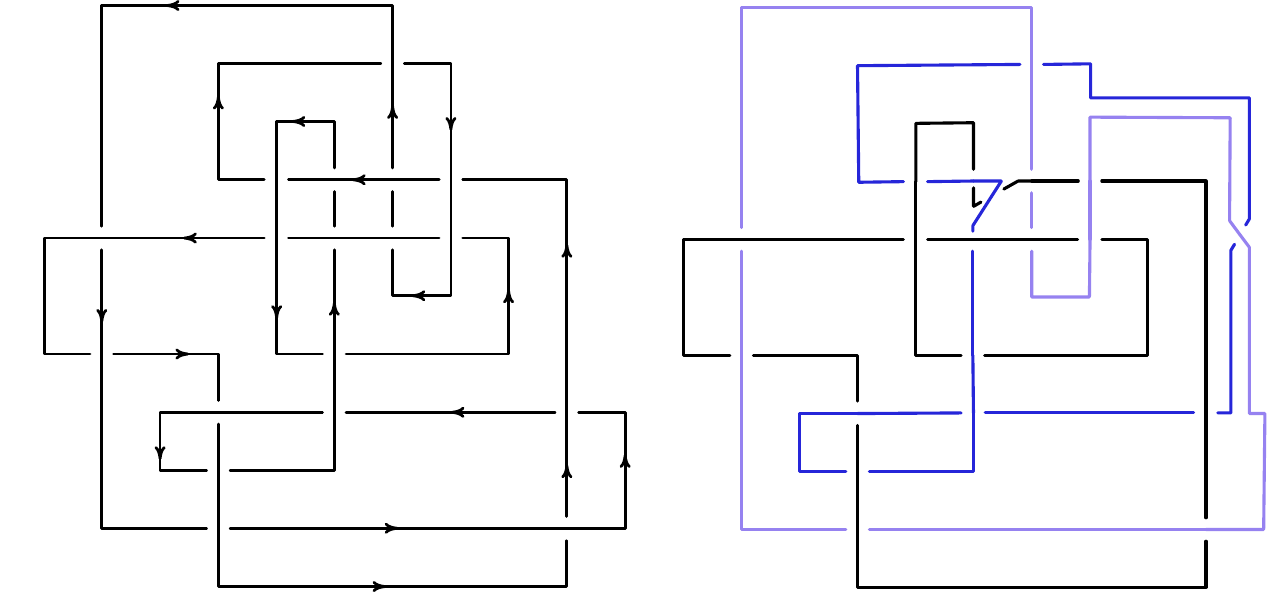}
    \caption{At left is the knot 18nh\_00159590.  At right is that knot with two bands added (one very short), to make a three component unlink. The corresponding ribbon disk shows that 18nh\_00159590 is a ribbon knot with fusion number at most two.}
    \label{fig:18nh_00159590}
\end{figure}



Interestingly, for the 800,000 ribbon knots built as symmetric unions with one twisting region containing a single half-twist, all of the reduced Khovanov homology ranks were $1\pmod 8$.  This suggests:

\begin{question}\label{qu:Kh sym unions 1 tw region}
If $K$ is a symmetric union with one twisting region (containing a single half-twist?), must we have
\[\rk\rKh_{\ring}(K)\equiv 1\pmod 8\]
for some (all?) $\ring$?
\end{question}

\subsection{Knot Floer homology computations}

All of our ribbon knots of up to 19 crossings had knot Floer homology ranks satisfying the Folk Conjecture \ref{conj:ribbon mod 8}\ref{it:r-HFK8} for knot Floer homology. 
However, from the sample of 800,000 knots built as symmetric unions, we found eight examples where the rank of the knot Floer homology was $5 \pmod{8}$, violating the Folk Conjecture \ref{conj:ribbon mod 8}\ref{it:r-HFK8} and thus completing the proof of Theorem \ref{thm:1 mod 8 false}.  One of these symmetric unions is depicted in Figure \ref{fig:eg1_symm_union_without_labels}, which was built from a 13-crossing diagram of $4_1$ as shown there.  The total rank of the knot Floer homology of this symmetric union is 77, which is $5 \pmod{8}$. The specific ranks broken down by grading are depicted in Figure \ref{fig:eg1_hfk_ranks}.

\begin{figure}
    \centering
    \includegraphics[width=0.7\textwidth]{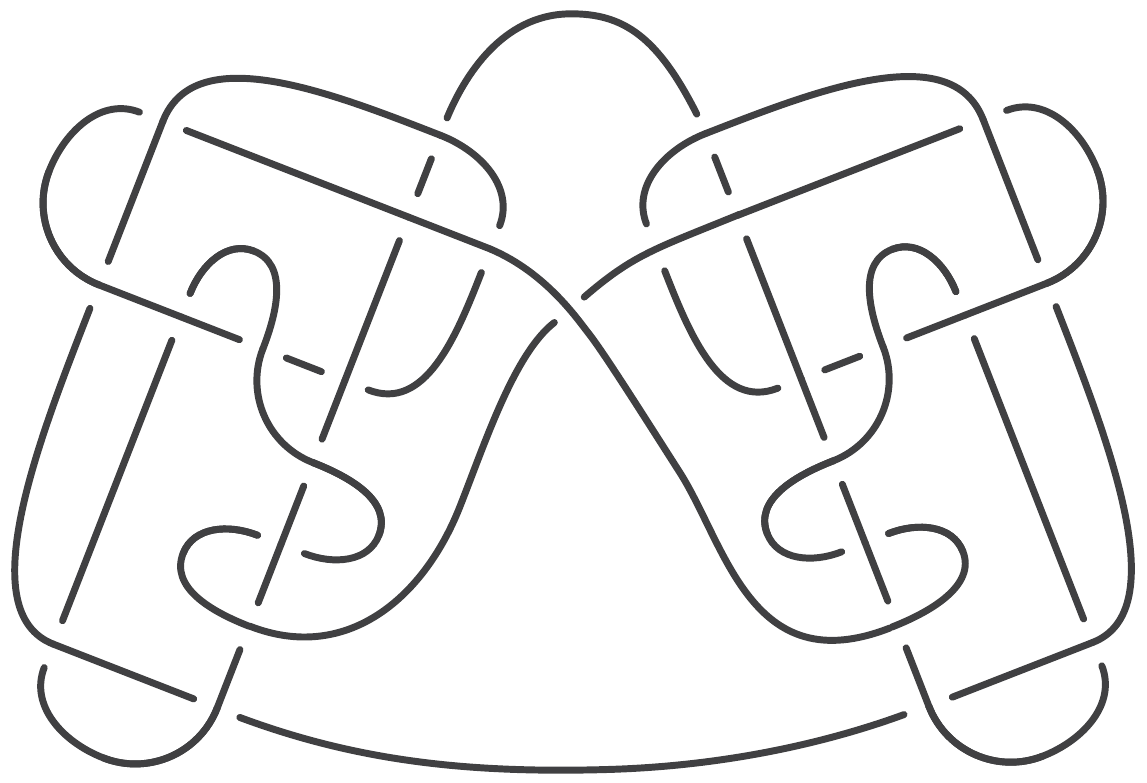}
    \caption{A symmetric union built from a 13-crossing diagram of the knot $4_1$, using one twisting region with one half-twist, which is the crossing in the middle of the diagram. The PD code for this knot is in Appendix~\ref{app: PD codes}.}
    \label{fig:eg1_symm_union_without_labels}
\end{figure}

\begin{figure}
    \centering
    \includegraphics[width=0.9\textwidth]{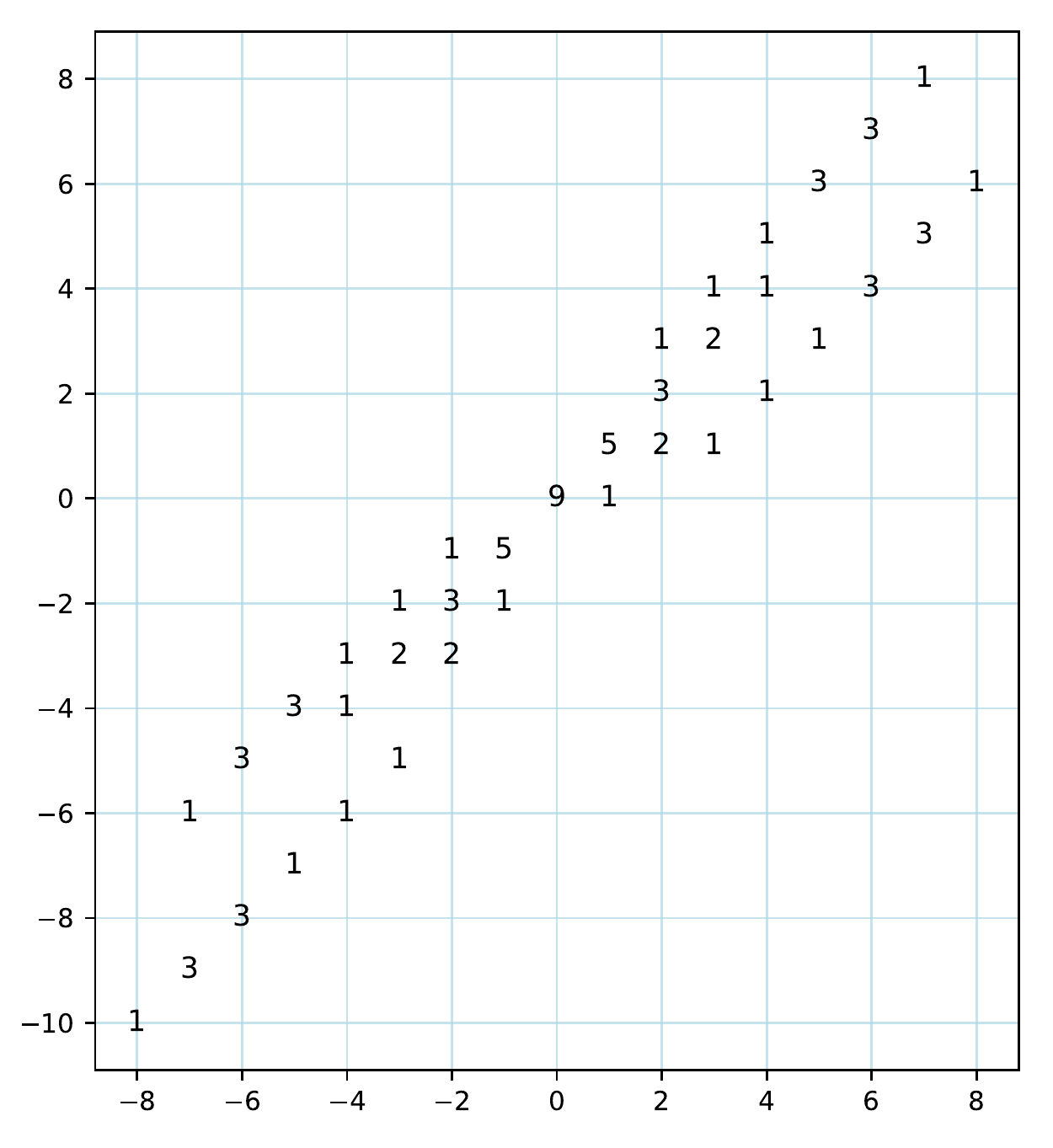}
    \caption{A plot of the knot Floer homology of the symmetric union depicted in Figure \ref{fig:eg1_symm_union_without_labels}. The horizontal axis is the Alexander grading and the vertical axis is the Maslov grading.}
    \label{fig:eg1_hfk_ranks}
\end{figure}

We found seven other knots built as symmetric unions in which the knot Floer homology had rank $5 \pmod{8}$. They are listed in Table \ref{tab:HFK counter-examples}.

\begin{table}[]
    \centering
\begin{tabular}{|c c |} 
 \hline
  \rule{0pt}{3ex}Original knot & $\rk\HFKh$ of the symmetric union\\ [0.5ex] 
 \hline\hline
\rule{0pt}{3ex}K4a1 & 77 \\  [0.5ex] 
 \hline
\rule{0pt}{3ex}K3a1 & 2733 \\  [0.5ex] 
 \hline
\rule{0pt}{3ex}K13n2191 & 9357 \\  [0.5ex] 
 \hline
\rule{0pt}{3ex}K4a1 & 149 \\  [0.5ex] 
 \hline
\rule{0pt}{3ex}K5a1 & 117 \\  [0.5ex] 
 \hline
\rule{0pt}{3ex}K6a3 & 661 \\  [0.5ex] 
 \hline
\rule{0pt}{3ex}K15n52125 & 21269 \\  [0.5ex] 
 \hline
\rule{0pt}{3ex}K15n52125 & 24989 \\  [0.5ex] 
 \hline
\end{tabular}
    \vspace{1ex}
    \caption{The ranks of the knot Floer homologies of the symmetric unions we found that violated the Folk Conjecture \ref{conj:ribbon mod 8}\ref{it:r-HFK8}.  All of these ranks are $5\pmod 8$, satisfying Conjecture \ref{conj:ribbon}\ref{it:r-HFK}. In each case, the original diagram had between 13 and 15 crossings; see \cite{AncillaryFiles} for the PD codes for these knots and related data.
    }
    \label{tab:HFK counter-examples}
\end{table}

\section{Ribbon knots with fusion number 1}
\label{sec:fusion1}

\subsection{Ranks of $\HFKh$ for torsion order 1 knots}
\label{sec:HFK tor order 1}

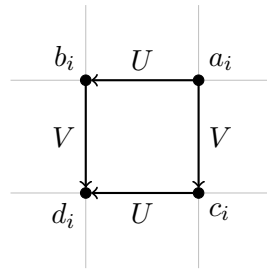
\begin{figure}
\begin{center}
  \def\step{1.5}
  \def\mini{0.07}
  \def\diag{0.05}
  \def\bld{1}
  \def\gsize{1}
  
  \begin{tikzpicture}
  
    \draw[step=\step cm, lightgray, very thin] (-\bld,-\bld) grid (\gsize*\step+\bld,\gsize*\step+\bld);
    
    \draw[fill=black] (\step,\step) circle (\mini);
    \draw[fill=black] (0,\step) circle (\mini);
    \draw[fill=black] (\step,0) circle (\mini);
    \draw[fill=black] (0,0) circle (\mini);
    
    \draw (\step,\step) node[anchor = south west] {$a_i$};
    \draw (0,\step) node[anchor = south east] {$b_i$};
    \draw (\step,0) node[anchor = north west] {$c_i$};
    \draw (0,0) node[anchor = north east] {$d_i$};
    
    \draw[->, thick] (\step-\mini,\step) to (\mini,\step);
    \draw (\step/2,\step) node[anchor = south] {$U$};
    
    \draw[->, thick] (\step,\step-\mini) to (\step,\mini);
    \draw (\step,\step/2) node[anchor = west] {$V$};
    
    \draw[->, thick] (\step-\mini,0) to (\mini,0);
    \draw (+\step/2,0) node[anchor = north] {$U$};
    
    \draw[->, thick] (0,\step-\mini) to (0, \mini);
    \draw (0,\step/2) node[anchor = east] {$V$};
    
  \end{tikzpicture}
\caption{A unit box in the $U-V$ plane.}
\label{fig:unitbox}
\end{center}
\end{figure}

In this section, we prove that ribbon knots with fusion number $1$ satisfy the stronger 1 mod 8 Folk Conjecture \ref{conj:ribbon mod 8}\ref{it:r-HFK}.  We will actually prove a more general statement, which involves the following quantity. Recall that the \emph{knot Floer torsion order} $\Ord(K)$ of a knot $K$ is the minimum power of $U$ that annihilates the torsion $\F[U]$-submodule of $\HFKm(K)$.  As a consequence of \cite[Corollary 1.7]{JMZ:ribbon}, ribbon knots with fusion number $1$ have torsion order $1$. Thus, Theorem \ref{thm:fusion1} from the introduction is a direct consequence of:

\begin{theorem}
\label{thm:TO}
If $K$ is a ribbon knot with $\Ord(K) = 1$, then
\[
\rk \HFKh(K) \equiv 1 \pmod8.
\]
\end{theorem}

\begin{proof}
By \cite[Proposition 3.6]{HKP:ribbon}, the knot Floer complex $\CFK_{\F[U,V]}(K)$ of a fusion number $1$ ribbon knot $K$ decomposes as
\[
\CFK_{\F[U,V]}(K) \simeq \F[U,V] \oplus A_1 \oplus \cdots \oplus A_\ell,
\]
where each $A_i$ is a free $\F[U,V]$-module of rank $4$, with 4 homogeneous generators $a_i, b_i, c_i, d_i$, and differential given by
\begin{align*}
    \de a_i &= U\cdot b_i + V \cdot c_i;\\
    \de b_i &= V \cdot d_i;\\
    \de c_i &= U \cdot d_i;\\
    \de d_i &= 0.
\end{align*}
(Hom-Kang-Park call each $A_i$ a \emph{unit box}. See Figure \ref{fig:unitbox} for an illustration.)
Each unit box contributes $4$ to the rank of $\HFKh(K)$, so we immediately see that $\rk\HFKh(K) \equiv 1 \pmod 4$. In order to get the stronger statement $\pmod8$, we use the Arf invariant.

Recall that the full knot Floer complex comes with two gradings, denoted $\gr_\w$ and $\gr_\z$ \cite{Z:gradings}. The $\delta$ grading is defined as
\[
\delta = \frac{\gr_\w + \gr_\z}2
\]
Since $\delta(\de) = \delta(U) = \delta(V) = -1$, it is straightforward to check that
\[
\delta(a_i) = 
\delta(b_i) = 
\delta(c_i) = 
\delta(d_i).
\]
Thus, the $\delta$-graded knot Floer homology of $K$ can be written as
\begin{equation}
\label{eq:HFKh-delta}
\HFKh(K) = \F_{\delta_0} \oplus \F^4_{\delta_1} \oplus \cdots \oplus \F^4_{\delta_\ell},
\end{equation}
for some $\delta$-gradings $\delta_0, \delta_1, \ldots, \delta_\ell \in \Z$.

The $\delta$-graded Euler characteristic of $\HFKh(K)$ is $\pm\det K$. Thus, from Equation \ref{eq:HFKh-delta}, we compute
\[
\pm\det K = (-1)^{\delta_0} + 4 \cdot (-1)^{\delta_1} + \cdots + 4 \cdot (-1)^{\delta_\ell}.
\]
Reducing modulo $8$ we get
\[
\det K \equiv 4 \ell \pm 1 \pmod8.
\]
By comparing with Levine's formula (cf.\ Theorem \ref{thm:Levine}) and using the fact that $\Arf K = 0$ (since $K$ is slice), we obtain that $\ell$ must be even. Therefore, by Equation \ref{eq:HFKh-delta}, we get that $\rk \HFKh(K) \equiv 1 \pmod 8$.
\end{proof}

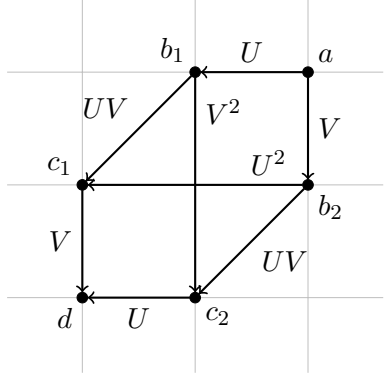
\begin{figure}
\begin{center}
  \def\step{1.5}
  \def\mini{0.07}
  \def\diag{0.05}
  \def\gwidth{2.5}
  
  \begin{tikzpicture}
  
    \draw[step=\step cm, lightgray, very thin] (-\gwidth,-\gwidth) grid (\gwidth,\gwidth);
    
    \draw[fill=black] (\step,\step) circle (\mini);
    \draw[fill=black] (0,\step) circle (\mini);
    \draw[fill=black] (\step,0) circle (\mini);
    \draw[fill=black] (-\step,0) circle (\mini);
    \draw[fill=black] (0,-\step) circle (\mini);
    \draw[fill=black] (-\step,-\step) circle (\mini);
    
    \draw (\step,\step) node[anchor = south west] {$a$};
    \draw (0,\step) node[anchor = south east] {$b_1$};
    \draw (\step,0) node[anchor = north west] {$b_2$};
    \draw (-\step,0) node[anchor = south east] {$c_1$};
    \draw (0,-\step) node[anchor = north west] {$c_2$};
    \draw (-\step,-\step) node[anchor = north east] {$d$};
    
    \draw[->, thick] (\step-\mini,\step) to (\mini,\step);
    \draw (\step/2,\step) node[anchor = south] {$U$};
    
    \draw[->, thick] (\step,\step-\mini) to (\step,\mini);
    \draw (\step,\step/2) node[anchor = west] {$V$};
    
    \draw[->, thick] (-\diag, \step-\diag) to (-\step+\diag, \diag);
    \draw (-\step/2,\step/2) node[anchor = south east] {$UV$};
    
    \draw[->, thick] (\step-\diag, -\diag) to (\diag, -\step+\diag);
    \draw (\step/2,-\step/2) node[anchor = north west] {$UV$};
    
    \draw[->, thick] (0,\step-\mini) to (0,-\step+\mini);
    \draw (0,0.65*\step) node[anchor = west] {$V^2$};
    
    \draw[->, thick] (\step-\mini,0) to (-\step+\mini,0);
    \draw (0.65*\step,0) node[anchor = south] {$U^2$};
    
    \draw[->, thick] (-\mini,-\step) to (-\step+\mini,-\step);
    \draw (-\step/2,-\step) node[anchor = north] {$U$};
    
    \draw[->, thick] (-\step,-\mini) to (-\step, \mini-\step);
    \draw (-\step,-\step/2) node[anchor = east] {$V$};
    
  \end{tikzpicture}
\caption{The $\F[U,V]$-chain complex $A$ in the $U-V$ plane.}
\label{fig:acyclicbad}
\end{center}
\end{figure}

\begin{remark}\label{rmk:HFK = Det mod 16}
In fact, Theorem \ref{thm:fusion1} can be further refined to prove that fusion number 1 ribbon knots $K$ satisfy
\[\rk\HFKh(K) \equiv \det K \pmod{16}.\]
The argument, which was pointed out to us in correspondence with Adam Levine, uses involutive Heegaard Floer techniques (specifically, the proof of Proposition 8.1 in \cite{HM:involutive}) and in fact shows the stronger result that torsion order 1 knots satisfy
\[\rk\HFKh(K) \equiv (-1)^{\tau(K)} \cdot \det K \pmod{16}.\]
\end{remark}

\begin{remark}
\label{rem:nohigherTO}
The strategy used in the proof of Theorem \ref{thm:TO}, which relies on the algebraic structure of acyclic summands of $\CFK_{\F[U,V]}$, must fail for knots with higher torsion order as seen by the counter-examples of Table \ref{tab:HFK counter-examples}.  Indeed this approach does not seem to shed light on the weaker Conjecture \ref{conj:ribbon}\ref{it:r-HFK} for higher torsion order knots either.  For example, $\CFK_{\F[U,V]}$ of a knot with torsion order $2$ might have a summand $A$ freely generated over $\F[U,V]$ by $a, b_1, b_2, c_1, c_2, d$, with gradings $\gr = (\gr_\w, \gr_\z)$ and differential given by
\begin{align*}
    \gr (a) &= (0,0);
    & \de a &= U\cdot b_1 + V \cdot b_2;\\
    \gr (b_1) &= (1,-1);
    & \de b_1 &= UV\cdot c_1 + V^2 \cdot c_2;\\
    \gr (b_2) &= (-1,1);
    & \de b_2 &= U^2\cdot c_1 + UV \cdot c_2;\\
    \gr (c_1) &= (2,0);
    & \de c_1 &= V \cdot d;\\
    \gr (c_2) &= (0,2);
    & \de c_2 &= U\cdot d;\\
    \gr (d) &= (1,1);
    & \de d &= 0.
\end{align*}
See Figure \ref{fig:acyclicbad} for an illustration.
The chain complex $A$ satisfies the formal symmetries of $\HFK$, and its contribution to $\HFKh$ would be a rank $6$ summand. Moreover, the resulting $\delta$-graded Euler characteristic would be $0$, so $A$ would not affect the Arf invariant. Thus, we do not know any algebraic obstructions to having a torsion order $2$ ribbon knot $K$ with
\[
\CFK_{\F[U,V]}(K) \cong \F[U,V] \oplus A,
\]
and this would not satisfy Conjecture \ref{conj:ribbon}\ref{it:r-HFK}.
\end{remark}

\begin{remark}
\label{rem:noKh}
It is also unlikely that the strategy used in the proof of Theorem \ref{thm:TO} may be adapted to (reduced) Khovanov homology in the hopes of answering Question \ref{qu:Khr fusion number 1}. The key ingredient of such a proof is Hom's decomposition \cite[Theorem 1]{H:survey}, and no analogue for Khovanov homology is presently known.

Furthermore, each unit box appearing in Hom's decomposition contributes to a summand $B$ of $\HFKh$ with gradings $(\gr_\w, \gr_\z, \Alex)$ being, up to an overall shift,
\[
B = \F_{(-1, 1, -1)} \oplus \F^2_{(0,0,0)} \oplus \F_{(1,-1,1)}.
\]
(Recall that the Alexander grading $\Alex$ is recovered by $\Alex = \frac{\gr_\w-\gr_\z}2$.)  See Figure \ref{fig:Bcomplex}.

However, $\rKh_{\ring}(6_1)$ (where $\ring$ is $\Z$ or any field $\F_p$) does not admit a bigraded decomposition into a 1-dimensional summand and some summands with a 1-2-1 pattern as in $B$. See Figure \ref{fig:StevedorerKh}. Note that $6_1$, a.k.a.\ Stevedore's knot, is ribbon with fusion number $1$.
\end{remark}

\begin{figure}
\begin{center}
  \def\step{1}
  \def\mini{0.07}
  \def\diag{0.05}
  \def\displment{0.9}
  \def\gwidth{2.9}
  
  \begin{tikzpicture}
  
    \draw[step=\step cm, lightgray, very thin] (-\gwidth+\step,-\gwidth+\step+\displment) grid (\gwidth-\displment,\gwidth);
    \draw[lightgray, very thin] (-\gwidth+\step, \gwidth) -- (-\gwidth+\step+\displment, \gwidth-\displment);
    \draw (-\gwidth+\step+0.2*\displment, \gwidth-0.2*\displment) node[anchor=west] {\footnotesize $\gr_\w$};
    \draw (-\gwidth+\step+0.2*\displment, \gwidth-0.4*\displment) node[anchor=north] {$\Alex$};
    
  \begin{scope}[xshift=0.5*\step cm, yshift=0.5*\step cm]
  
    \draw (0,0) node {$\F^2$};
    \draw (\step, \step) node {$\F$};
    \draw (-\step, -\step) node {$\F$};
    
    \draw (-\step,2*\step) node {$-1$};
    \draw (0,2*\step) node {$0$};
    \draw (\step,2*\step) node {$1$};
    \draw (-2*\step,-\step) node {$-1$};
    \draw (-2*\step,0) node {$0$};
    \draw (-2*\step,\step) node {$1$};
    
\end{scope}
  \end{tikzpicture}
\caption{The $(\gr_\w, \Alex)$-bigraded $\F$-vector space $B$ in the $\gr_\w-\Alex$ plane.}
\label{fig:Bcomplex}
\end{center}
\end{figure}

\begin{figure}
\begin{center}
  \def\step{1}
  \def\mini{0.07}
  \def\diag{0.05}
  \def\displment{0.9}
  \def\gwidth{4.9}
  
  \begin{tikzpicture}
  
    \draw[step=\step cm, lightgray, very thin] (-\gwidth+\step,-\gwidth+\step+\displment) grid (\gwidth-\displment,\gwidth);
    \draw[lightgray, very thin] (-\gwidth+\step, \gwidth) -- (-\gwidth+\step+\displment, \gwidth-\displment);
    \draw (-\gwidth+\step+0.4*\displment, \gwidth-0.2*\displment) node[anchor=west] {$h$};
    \draw (-\gwidth+\step+0.2*\displment, \gwidth-0.4*\displment) node[anchor=north] {$q$};
    
  \begin{scope}[xshift=1.5*\step cm, yshift=1.5*\step cm]
  
    \draw (2*\step, 2*\step) node {$\Z$};
    \draw (\step, \step) node {$\Z$};
    \draw (0,0) node {$\Z^2$};
    \draw (-\step, -\step) node {$\Z^2$};
    \draw (-2*\step, -2*\step) node {$\Z$};
    \draw (-3*\step, -3*\step) node {$\Z$};
    \draw (-4*\step, -4*\step) node {$\Z$};
    
    \foreach \x in {-4,...,2}
        \draw (\x*\step,3*\step) node {$\x$};
    
    \foreach \x in {-4,...,2}
    {   \pgfmathtruncatemacro{\y}{2*\x};
        \draw (-5*\step,\x*\step) node {$\y$};
    }
    
\end{scope}

  \end{tikzpicture}
\caption{The $(h, q)$-bigraded abelian group $\rKh_{\Z}(6_1)$ in the $h-q$ plane.
Note that since $\rKh_{\Z}(6_1)$ is a free abelian group one can recover $\rKh_{\F_p}(6_1)$ just by tensoring $\rKh_{\Z}(6_1)$ with $\F_p$.}
\label{fig:StevedorerKh}
\end{center}
\end{figure}

\subsection{Ranks of $\rKh$ for $X$-torsion order 1 knots}\label{sec:Kh xo 1}
In this section, we present a proof, due to Robert Lipshitz and Sucharit Sarkar, of Theorem \ref{thm:fusion1 Kh} stating that ribbon knots with fusion number 1 satisfy Conjecture \ref{conj:ribbon}\ref{it:r-rKh}.  As in Section \ref{sec:HFK tor order 1}, the proof applies more generally to a class of knots defined in terms the following:

\begin{definition}[\cite{SS:xo,AL:xo}]
The \emph{$\ring$-extortion order} of a knot $K$, denoted $\xo_\ring(X)$, is the maximal order of $X$-torsion in a certain lift of Lee homology $LH(K)$ (obtained by setting $X^2=t$ in the Frobenius algebra defining Khovanov homology, instead of $X^2=1$; see \cite{Khovanov2006}) viewed as a module over $\ring[X]$.
\end{definition}

Lipshitz and Sarkar used this notion to prove the following theorem, which we will see is a generalization of Theorem \ref{thm:fusion1 Kh}; we include the unpublished proof with their permission.

\begin{theorem}[Lipshitz-Sarkar]\label{thm:xo 1 Kh}
If $\ring$ is a field and $K$ is a ribbon knot with $\xo_\ring(K)=1$, then $\rk\rKh_\ring(K)\equiv 1\pmod 4$.
\end{theorem}

\begin{proof}
The Lee complex $LC(K)$, viewed as a module over the PID $\ring[X]$, must split as
\[LC(K)\cong \ring[X]^r \oplus (\ring[X]\xrightarrow{X^{a_1}}\ring[X]) \oplus \cdots \oplus (\ring[X]\xrightarrow{X^{a_k}}\ring[X]).\]
If $\xo_\ring(K)=1$, we must have 
\[r=a_1=\cdots =a_k =1\]
in the above.

The reduced Khovanov complex of $K$ is recovered from this Lee complex by setting $X=0$.  In this case, each summand $(\ring[X]\xrightarrow{0}\ring[X])$ contributes a two-dimensional vector space $V$ to the total reduced Khovanov homology.
Consider a choice of mixed grading $\delta=\frac{q}{2}-h$ for Khovanov homology (here $q$ denotes the usual quantum grading, and $h$ the usual homological grading) allowing the graded Euler charateristic to recover the determinant of the knot (see \cite{MO:quasi-alternating}).  Since $a_i = 1$, this $V$ must be contained in a single $\delta$-grading.

Now we let $k_e$ (respectively $k_o$) denote the number of copies of $V$ in even (respectively odd) $\delta$-grading.  Taking the graded Euler characteristic then tells us that
\[1+2k_e - 2k_o \equiv \det(K) \equiv 1 \pmod 8,\]
and therefore $k_e\equiv k_o \pmod 4$.  This implies that the total rank is
\[1+2k_e+2k_o \equiv 1 \pmod 4.\qedhere\]
\end{proof}

\begin{proof}[Proof of Theorem \ref{thm:fusion1 Kh}]
Suppose $K$ is a ribbon knot of fusion number 1, and $\ring=\Q$ or $\F_p$ for $p\neq 2$.  Then \cite[Corollary 1.2]{SS:xo} shows that $\xo_\ring(K)=1$, and thus Theorem \ref{thm:xo 1 Kh} completes the proof.
\end{proof}

\section{Non-locality}
\label{sec:non-locality}

\begin{figure}
    \centering
    \input{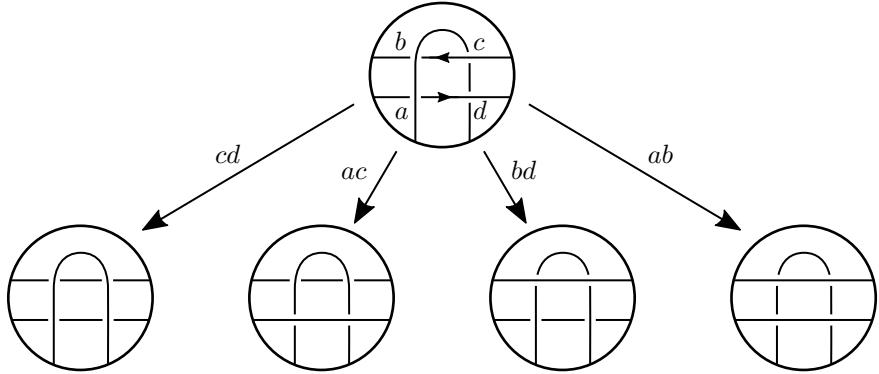}
    \caption{Given a ribbon knot $K$ and a ribbon disc $D$, there exists a projection of $K$ that locally looks as the tangle at the top of the picture and such that the restriction of $D$ to the tangle is the obvious locally ribbon surface. The resulting ribbon singularity can be resolved with 2 crossing changes in 4 different ways.}
    \label{fig:ribbon-singularity}
\end{figure}

A tempting approach to prove Conjecture \ref{conj:ribbon} is to induct on the number of ribbon singularities, and hope to relate the ranks mod $4$ of a ribbon knot $K$ with the rank modulo $4$ of simpler ribbon knots.
Every ribbon knot $K$ together with a ribbon disc $D$ has a local projection as in Figure \ref{fig:ribbon-singularity}, where a ribbon singularity of $D$ is visible.
One can change two of the four local crossings in four different ways to resolve the ribbon singularity, and get a new ribbon knot together with a ribbon disc with one fewer ribbon singularity.

One might hope to find a local proof of Conjecture \ref{conj:ribbon}, based on this local move. The next statement shows that this cannot happen unless the ribbon disc is used in an essential way in the induction step.

\begin{proposition}
\label{prop:non-locality}
There exist knots that can be turned into the unknot $U$ by a sequence of local moves as in Figure \ref{fig:ribbon-singularity} and that do not satisfy the equalities in Conjecture \ref{conj:ribbon}.
\end{proposition}
\begin{proof}
Figure \ref{fig:non-locality} shows a projection of the knot $6_2$ with a local tangle that looks exactly as in Figure \ref{fig:ribbon-singularity} (orientation included).
Any of the 4 simplifications described in Figure \ref{fig:ribbon-singularity} turns $6_2$ into the unknot.
The knot $6_2$ is alternating and has determinant $11$. Thus, by \cite{MO:quasi-alternating}, its knot Floer homology and reduced Khovanov homology (over any ring $\ring$) have rank $11$, which is $3 \pmod4$. Of course, $6_2$ is not ribbon (nor slice).
\end{proof}

\begin{figure}
    \centering
    \includegraphics{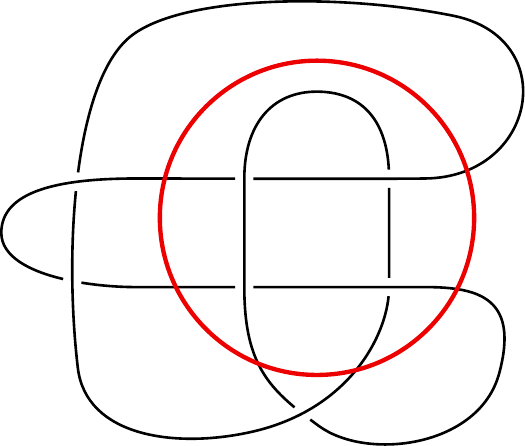}
    \caption{The knot $6_2$ can be unknotted with the local move in Figure \ref{fig:ribbon-singularity}, but its $\HFKh$ and $\rKh$ do not have rank congruent to $1 \pmod 4$.}
    \label{fig:non-locality}
\end{figure}

\appendix
\section{Polynomial invariants and the Arf invariant}
\label{appendix}

\subsection{Arf invariants}
A classical knot invariant that is famously related to modulo $8$ reductions and to the Alexander and Jones polynomials is the Arf invariant.
Recall that $\Arf{K}$ is a concordance invariant with values in $\Z/2\Z$, which can be recovered from the signed determinant of the knot via Levine's formula:
\begin{theorem}[{\cite[p.\ 544]{L:Arf}}]
\label{thm:Levine}
For any knot $K \subset S^3$,
\[\det K \equiv 4 \cdot \Arf K + 1 \pmod8.\]
\end{theorem}
Recall that $\det K$ is obtained from the (Conway-normalised) Alexander or the Jones polynomial via the identities $\det K = \Delta_{K}(-1) = V_K(-1)$. If one uses $|\det K|$ instead, then the $+1$ in Levine's formula should be replaced by $\pm1$.

\begin{question}
\label{qn:ArfQ}
Let $H$ denote knot Floer homology or reduced Khovanov homology (over $\mathbb{F}_p$ or $\Z$). Is it true that for every knot $K \subset S^3$
\begin{equation}
\label{eq:ArfQ}
\rk H(K) \equiv 4 \cdot \Arf K \pm 1 \pmod8?
\end{equation}
\end{question}

Note that Equation \eqref{eq:ArfQ} holds for all quasi-alternating knots, by Levine's formula, and for all untwisted Whitehead doubles when $H = \HFKh$ (see Section \ref{sec:Whitehead}).

\begin{answer}\label{no Arfstyle rk mod 8}
No.  The counter-examples of Theorem \ref{thm:1 mod 8 false} also work here.
\end{answer}

\begin{remark}\label{rmk:no adjusted Arf either}
In fact the Knight Move knot from \cite{MM:knightmove} also gives a counter-example for Question \ref{qn:ArfQ}.  However, Theorem \ref{thm:1 mod 8 false} is actually showing a stronger statement: namely, that no concordance invariant can `fit' into Equation \eqref{eq:ArfQ} in place of $\Arf(K)$.
\end{remark}

\subsection{The reduction formula for the Jones polynomial}

Another way the Arf invariant can be recovered from the (symmetrised) Alexander polynomial is via Robertello's reduction formula below.

\begin{theorem}[{\cite[Theorem 6]{R:Arf}}]
\label{thm:R6}
For any knot $K \subset S^3$, 
\begin{itemize}
\item $\Arf(K) = 0$ if and only if $\Delta_{K}(t) \equiv 1 \pmod{2, 1+t^4}$;
\item $\Arf(K) = 1$ if and only if $\Delta_{K}(t) \equiv t^{-1}+1+t \pmod{2, 1+t^4}$.
\end{itemize}
\end{theorem}

Since we did not find the analogous result for the Jones polynomial in the literature, we state and prove it here.

\begin{theorem}
\label{thm:R6Jones}
For any knot $K \subset S^3$, 
\begin{itemize}
\item $\Arf(K) = 0$ if and only if $\V_{K}(t) \equiv 1 \pmod{2, 1+t^4}$;
\item $\Arf(K) = 1$ if and only if $\V_{K}(t) \equiv t^{-1}+1+t \pmod{2, 1+t^4}$.
\end{itemize}
\end{theorem}

Recall that the Jones polynomial $\Jones_L(t)$ of a link $L$ satisfies the skein relation
\begin{equation}
\label{eq:skeinJones}
t^{-1} \cdot \J_{L_+}(t) - t \cdot \J_{L_-}(t) = (t^{1/2} - t^{-1/2}) \cdot \J_{L_0}(t),
\end{equation}
where $L_0$ denotes the oriented resolution at the given crossing.

If $L_+$ is a knot, then
\begin{equation}
\label{eq:skeinArf}
\Arf{K_+} + \Arf{K_-} \equiv \lk{L_0} \pmod 2,
\end{equation}
where $\lk{L_0}$ denotes the linking number of the 2 components of $L_0$.

The following theorem immediately implies Theorem \ref{thm:R6Jones}.

\begin{theorem}
\label{thm:R6JonesX}
For any knot $K \subset S^3$, 
\begin{itemize}
\item $\Arf(K) = 0$ if and only if $\V_{K}(t) \equiv 1 \pmod{2, 1+t^4}$;
\item $\Arf(K) = 1$ if and only if $\V_{K}(t) \equiv t^{-1}+1+t \pmod{2, 1+t^4}$.
\end{itemize}
For any 2-component link $L \subset S^3$,
\begin{itemize}
\item $\lk(L) \equiv 0 \pmod2$ if and only if $t^{1/2} \cdot \V_{L}(t) \equiv u^2 \cdot (1+t) \pmod{2, 1+t^4}$ for some invertible $u \in \Z[t, t^{-1}]$;
\item $\lk(L) \equiv 1 \pmod2$ if and only if $t^{1/2} \cdot \V_{L}(t) \equiv u \cdot (1+t^2) \pmod{2, 1+t^4}$ for some invertible $u \in \Z[t, t^{-1}]$.
\end{itemize}
\end{theorem}

By $a \equiv b \pmod{2, 1+t^4}$ we mean that the images of $a$ and $b$ are the same in the ring $\Z[t,t^{-1}]/(2, 1+t^4)$.

\begin{lemma}
\label{lem:1+t}
The linear map
\[
\phi_{1+t} \colon \F_2[t]/(1+t^4) \to \F_2[t]/(1+t^4)
\]
given by multiplication by $1+t$ has rank 3, and its kernel is generated by $1+t+t^2+t^3$.
\end{lemma}
\begin{proof}
The proof is a straightforward computation in linear algebra.
\end{proof}

\begin{proof}[Proof of Theorem \ref{thm:R6JonesX}]
We argue by induction on the number of crossings of a diagram $D$ for the knot or link at hand. If there are no crossings, then $D$ is a diagram of the unknot or the 2-component unlink, and we are done.

Suppose now that the theorem holds for all knots or 2-component links with a diagram with $< n$ crossings. First we will prove that the theorem holds for all knots with a diagram with $n$ crossings, then we will move to 2-component links.

For every knot diagram $D$ with $n$ crossings, there is a sequence of crossing changes that turns $D$ into a diagram of the unknot; thus, there is a finite sequence of diagrams $D_0, D_1, \ldots, D_k = D$ such that $D_0$ is a diagram for the unknot, and $D_i$ and $D_{i+1}$ differ only for a crossing change.
We nest an argument by induction on $k$: the theorem holds for the knot represented by the diagram $D_0$ (i.e.\ the unknot).
For the induction step suppose that $L_+$ and $L_-$ are the two knots associated to $D_{k-1}$ and $D_{k}$ (not necessarily in this order), and let $L_0$ be the link associated to the oriented smoothing of the projection. Note that $L_0$ has a diagram with $<n$ crossings. Suppose that $\lk(L_0) = 0$, so $\Arf{L_+}=\Arf{L_-}$. By the skein relation \eqref{eq:skeinJones} and the main induction hypothesis, after simplifying we get
\[
t^{-1} \cdot \V_{L_+} + t \cdot \V_{L_-} \equiv t+t^3 \pmod{2, 1+t^4}.
\]
Solving e.g.\ for $\V_{L_+}$ yields
\begin{equation}
\label{eq:solvingforVL_+}
\V_{L_+} \equiv t^2 \cdot \V_{L_-} + 1+t^2 \pmod{2, 1+t^4}.
\end{equation}
From here it is easy to check that
\begin{itemize}
\item $\V_{L_-}(t) \equiv 1 \pmod{2, 1+t^4}$ implies $\V_{L_+}(t) \equiv 1 \pmod{2, 1+t^4}$, and
\item $\V_{L_-}(t) \equiv t^{-1}+1+t \pmod{2, 1+t^4}$ implies $\V_{L_+}(t) \equiv t^{-1}+1+t \pmod{2, 1+t^4}$.
\end{itemize}
Solving for $\V_{L_-}$ yields the same equation as \eqref{eq:solvingforVL_+} with $\V_{L_-}$ and $\V_{L_+}$ swapped. Thus, it does not matter in which order $L_-$ and $L_+$ are associated to $D_{k-1}$ and $D_{k}$.

Suppose instead that $\lk(L_0) = 1$, so $\Arf{L_+} \neq \Arf{L_-}$. Arguing as above (solving e.g.\ for $V_{L_+}$) we obtain
\begin{equation*}
\V_{L_+} \equiv t^2 \cdot \V_{L_-} + 1 + t + t^2 + t^3 \pmod{2, 1+t^4}.
\end{equation*}
From here we easily see that
\begin{itemize}
\item $\V_{L_-}(t) \equiv 1 \pmod{2, 1+t^4}$ implies $\V_{L_+}(t) \equiv t^{-1}+1+t \pmod{2, 1+t^4}$, and
\item $\V_{L_-}(t) \equiv t^{-1}+1+t \pmod{2, 1+t^4}$ implies $\V_{L_+}(t) \equiv 1 \pmod{2, 1+t^4}$.
\end{itemize}
The reasoning solving for $V_{L_-}$ is completely analogous.

We now turn our attention to 2-component links with $\leq n$ crossings.
If there are no crossings involving both components of a given link $L$, then the link is a split union $K_1 \sqcup K_2$ (so $\lk L = 0$), and its Jones polynomial is
\[
\V_L(t) = -(t^{1/2}+t^{-1/2})\cdot \V_{K_1}(t) \cdot \V_{K_2}(t).
\]
Thus, we obtain
\[
t^{1/2}\cdot \V_{L}(t) \equiv (1+t) \cdot \V_{K_1}(t) \cdot \V_{K_2}(t) \pmod{2, 1+t^4}.
\]
Both $K_1$ and $K_2$ have $\leq n$ crossings, so by inductive hypothesis the product $\V_{K_1}(t) \cdot \V_{K_2}(t)$ is either $1$ or $t^{-1} + 1 + t$. The former is obviously the square of an invertible element in $\F_2[t]/(1+t^4)$, whereas the latter can be replaced by $t^2$, since $(t^{-1} + 1 + t) + t^2 \in \ker \phi_{1+t}$ (see Lemma \ref{lem:1+t}). In all cases, we get the formula
\[
t^{1/2}\cdot \V_{L}(t) \equiv u^2 \cdot (1+t) \pmod{2, 1+t^4}.
\]

Finally, we consider the case of a 2-component link diagram $D$ with $\leq n$ crossings and at least one crossing between the two components.
As we did for knots, we nest an induction argument, this time on the number of crossing changes needed to turn $D$ into a diagram of a split link. The base of the induction is given by the preceding paragraph. For the induction step, pick a sequence of link diagrams $D_0, \ldots, D_k=D$ such that each one is related to the previous one by a crossing change involving both link components. 
Suppose that $D_k$ is a diagram for $L_+$ and $D_{k-1}$ is a diagram for $L_-$ (we will deal with the other case later). By solving the skein relation for $\V_{L_+}$ one gets
\begin{equation}
\label{eq:skeininproofforlinks}
t^{1/2} \cdot \V_{L_+}(t) \equiv t^2 \cdot (t^{1/2}\cdot \V_{L_-}(t)) + (t+t^2)\cdot \V_{L_0}(t) \pmod{2, 1+t^4}.
\end{equation}
Note that $\lk{L_+} + \lk{L_-} \equiv 1 \pmod 2$ and that $L_0$ is a knot.
If $\lk{L_-} \equiv 0 \pmod 2$ (hence $\lk{L_+} \equiv 1 \pmod 2$), the first summand in Equation \eqref{eq:skeininproofforlinks} is $1+t$ or $t^2+t^3$, while the second summand is $t + t^2$ or $1+t^3$ (depending on the Arf invariant of $L_0$). In all cases, one gets 
\[
t^{1/2} \cdot \V_{L_+}(t) \equiv 1+t^2 \textrm{ or } t+t^3,
\]
as desired.
If instead $\lk{L_-} \equiv 1 \pmod 2$ (hence $\lk{L_+} \equiv 0 \pmod 2$), the first summand in Equation \eqref{eq:skeininproofforlinks} is $1+t^2$ or $t+t^3$, while the second summand is $t + t^2$ or $1+t^3$ (again depending on the Arf invariant of $L_0$). In all cases, one gets 
\[
t^{1/2} \cdot \V_{L_+}(t) \equiv 1+t \textrm{ or } t^2+t^3,
\]
as desired.
Lastly, if $D_k$ is a diagram for $L_-$ and $D_{k-1}$ is a diagram for $L_+$, by solving the skein relation for $\V_{L_-}$ one gets
\begin{equation*}
t^{1/2} \cdot \V_{L_-}(t) \equiv t^2 \cdot (t^{1/2}\cdot \V_{L_+}(t)) + (1+t^3)\cdot \V_{L_0}(t) \pmod{2, 1+t^4}
\end{equation*}
instead of \eqref{eq:skeininproofforlinks}, and the proof is completely analogous.
\end{proof}

The following corollary is a Jones polynomial analogue of \cite[Theorem 7]{R:Arf}. Its proof is immediate from the statement of Theorem \ref{thm:R6JonesX}.

\begin{corollary}
For a knot $K \subset S^3$, let $\V_K(t) = \sum_{i} c_i t^i$. Then
\[
\Arf K \equiv \sum_{i \equiv 1 \pmod4} c_i \equiv \sum_{i \equiv -1 \pmod 4} c_i \pmod 2.
\]
\end{corollary}

\subsection{Lickorish's unifying statement for the Alexander polynomial}
Recall that the \emph{Conway potential} of a link $L$ is a polynomial
\[
\nabla_L(z) = a_0 + a_1 z + a_2 z^2 + \ldots
\]
If $K$ is a knot, then
\[
\nabla_K(z) = 1 + a_2 z^2 + a_4 z^4 + \ldots
\]
By substituting $z \mapsto t^{-1/2}-t^{1/2}$, one recovers the (Conway-normalised) Alexander polynomial
\[
\Delta_K(t) = \nabla_K(t^{-1/2}-t^{1/2}).
\]

Recall that the Arf invariant can be recovered from the Alexander polynomial in two ways, by Levine's formula (Theorem \ref{thm:Levine}) or Robertello's reduction formula (Theorem \ref{thm:R6}). It turns out that both formulas can be easily deduced from the following unifying result of Lickorish.

\begin{theorem}[{\cite[Theorem 10.7]{L:KnotTheory}}]
For a knot $K$, $a_2 \equiv \Arf(K) \pmod 2$, where $a_2$ is the coefficient of $z^2$ in the Conway potential.
\end{theorem}

\subsection{A unifying statement for the Jones polynomial?}
As in the case of the Alexander polynomial, we can recover the Arf invariant from the Jones polynomial by the reduction formula (cf.\ Theorem \ref{thm:R6Jones}), and also by the usual Levine formula (cf.\ Theorem \ref{thm:Levine}), using the fact that the determinant of a knot can be recovered by the Jones polynomial via the identity $\det K = \Jones_K(-1)$.

It is natural to ask whether there is a unifying statement as in the case of the Alexander polynomial. As an indication of the relation of the two statements, we note that the strategy of the proof of Theorem \ref{thm:R6JonesX} can be followed step by step, but using the skein relation
\[
\det{L_+} - \det{L_-} = 2i\det{L_0},
\]
to prove Levine's formula:
\[
\det K \equiv 4\Arf K +1 \pmod 8 \qquad \text{and} \qquad \det L \equiv 2 i \lk L \pmod 2,
\]
where $K$ is any knot and $L$ is any 2-component link.

For the Jones polynomial, there is an extra statement that relates to the Arf invariant, which holds not only for knots but also for links.

Recall that a link $L$ is \emph{proper} if for all component $K$ of $L$
\[
\lk(K, L \setminus K) \equiv 0 \pmod 2.
\]
In particular, any knot is a proper link.

For a proper link $L$, we denote its Arf invariant by $\Arf L$. Such an invariant is not defined for improper links.

\begin{theorem}[{\cite[Theorem 10.6]{L:KnotTheory}}]
\label{thm:LickorishArf}
For a link $L \subset S^3$,
\[
\V_L(i) =
\begin{cases}
(-\sqrt2)^{\# L -1} \cdot (-1)^{\Arf L} & \textrm{if $L$ is proper}\\
0 & \textrm{otherwise}\\
\end{cases}
\]
\end{theorem}

\begin{question}
Is there a unifying statement for Theorems \ref{thm:R6Jones}, \ref{thm:Levine}, and \ref{thm:LickorishArf}?
\end{question}

\section{PD codes of the counter-examples}
\label{app: PD codes}

Here are the planar diagram (PD) codes of the four knots in
Section~\ref{sec:Computations} where $\rk \rKh_\Q$ is $5 \pmod{8}$:

{\footnotesize
\begin{verbatim}
18nh_00159590:
  [[16, 2, 17, 1], [2, 7, 3, 8], [10, 3, 11, 4], [23, 4, 24, 5], 
   [5, 24, 6, 25], [6, 11, 7, 12], [27, 8, 28, 9], [9, 28, 10, 29], 
   [12, 32, 13, 31], [32, 14, 33, 13], [14, 36, 15, 35], [36, 16, 1, 15],
   [17, 21, 18, 20], [33, 19, 34, 18], [19, 35, 20, 34], [26, 21, 27, 22], 
   [22, 29, 23, 30], [30, 25, 31, 26]]
\end{verbatim}
 }

{\footnotesize
\begin{verbatim}
18nh_00752242:
  [[10, 2, 11, 1], [2, 26, 3, 25], [3, 32, 4, 33], [33, 4, 34, 5],
   [20, 5, 21, 6], [6, 23, 7, 24], [7, 15, 8, 14], [15, 9, 16, 8],
   [36, 10, 1, 9], [11, 17, 12, 16], [17, 13, 18, 12], [13, 19, 14, 18],
   [24, 19, 25, 20], [21, 28, 22, 29], [29, 22, 30, 23], [26, 31, 27, 32],
   [34, 27, 35, 28], [30, 35, 31, 36]]
\end{verbatim}
}

{\footnotesize
\begin{verbatim}
19nh_000129633:
  [[30, 2, 31, 1], [2, 21, 3, 22], [3, 33, 4, 32], [33, 5, 34, 4], 
   [5, 29, 6, 28], [6, 15, 7, 16], [16, 7, 17, 8], [8, 27, 9, 28],
   [36, 9, 37, 10], [25, 10, 26, 11], [11, 38, 12, 1], [12, 17, 13, 18],
   [18, 13, 19, 14], [14, 19, 15, 20], [29, 21, 30, 20], [22, 32, 23, 31],
   [34, 24, 35, 23], [24, 36, 25, 35], [37, 27, 38, 26]]
\end{verbatim}
}

{\footnotesize
\begin{verbatim}
19nh_000305767:
  [[1, 9, 2, 8], [11, 3, 12, 2], [3, 30, 4, 31], [21, 5, 22, 4], 
   [34, 5, 35, 6], [6, 24, 7, 23], [7, 15, 8, 14], [26, 10, 27, 9],
   [10, 28, 11, 27], [12, 31, 13, 32], [32, 13, 33, 14], [15, 36, 16, 37],
   [37, 16, 38, 17], [17, 38, 18, 1], [18, 25, 19, 26], [19, 29, 20, 28],
   [29, 21, 30, 20], [22, 34, 23, 33], [24, 35, 25, 36]]
\end{verbatim}
}

\noindent
Also, here is the PD code for the knot in
Figure~\ref{fig:eg1_symm_union_without_labels}.

{\footnotesize
\begin{verbatim}
  [[7, 37, 8, 36], [40, 34, 41, 33], [31, 13, 32, 12], [8, 16, 9, 15],
   [14, 6, 15, 5], [35, 7, 36, 6], [32, 40, 33, 39], [26, 47, 27, 48],
   [46, 19, 47, 20], [41, 24, 42, 25], [44, 3, 45, 4], [25, 20, 26, 21],
   [48, 27, 1, 28], [2, 43, 3, 44], [23, 4, 24, 5], [38, 12, 39, 11],
   [42, 45, 43, 46], [21, 28, 22, 29], [10, 30, 11, 29], [37, 31, 38, 30],
   [13, 35, 14, 34], [16, 10, 17, 9], [17, 22, 18, 23], [18, 1, 19, 2]]
\end{verbatim}
}

\noindent
The PD codes for the rest of the knots in Table~\ref{tab:HFK counter-examples} are in \cite{AncillaryFiles}.

%
%
%
\bibliographystyle{alphaurl}
\bibliography{bibliography}

\end{document}